\documentclass[12pt]{amsart}
\pdfoutput=1 

\usepackage[
text={440pt,575pt},
headheight=9pt,
centering
]{geometry}

\usepackage{hyperref}
\usepackage{xcolor}

\definecolor{darkred}{RGB}{160,0,0}
\definecolor{darkblue}{RGB}{0,0,160}
\hypersetup{
  colorlinks,
  citecolor=darkblue,
  filecolor=black,
  linkcolor=darkblue,
  urlcolor=darkblue
}

\usepackage[utf8]{inputenc}
\usepackage[T1]{fontenc}
\usepackage{microtype}
\usepackage{lmodern}
\usepackage{amsmath,amssymb,amsthm}
\usepackage{soul}
\usepackage{mdwlist}
\usepackage[pdftex]{graphicx}
\usepackage{latexsym}
\usepackage{verbatim}
\usepackage{graphicx,caption,subcaption}
\usepackage{xcolor}
\usepackage{todonotes} 
\usepackage{enumitem}
\usepackage[normalem]{ulem}

\allowdisplaybreaks 

\theoremstyle{plain}
\newtheorem{theorem}{Theorem}
\newtheorem{corollary}[theorem]{Corollary}

\theoremstyle{definition}
\newtheorem{definition}[theorem]{Definition} 	
\newtheorem{remark}[theorem]{Remark}	   
\newtheorem{example}[theorem]{Example}

\newcommand{\R}{\mathbb{R}}

\newcommand{\N}{\mathbb{N}}
\newcommand{\PP}{\mathbb{P}}
\newcommand{\NN}{\mathrm{N}}
\newcommand{\C}{\mathbb{C}}
\newcommand{\D}{\mathcal{D}}
\newcommand{\G}{\mathcal{G}}
\newcommand{\E}{\mathbb{E}}
\newcommand{\Var}{\mathrm{Var}}

\newcommand{\Bin}{\mathrm{Bin}}
\newcommand{\rk}{\mathrm{rk}}
\newcommand{\inv}{\mathrm{inv}}
\newcommand{\des}{\mathrm{des}}
\newcommand{\dx}{\mathrm{d}x}

\title{Extreme Values of Permutation Statistics}
\author{Philip Dörr, Thomas Kahle}
\date{\today}



\thanks{The authors are supported by the DFG (314838170, GRK 2297,
  ``MathCoRe'').}

\begin{document}

\begin{abstract}
We investigate extreme values of Mahonian and Eulerian distributions arising from counting inversions and descents of random elements of finite Coxeter groups. To this end, we construct a triangular array of either distribution from a sequence of Coxeter groups with increasing ranks. To avoid degeneracy of extreme values, the number of i.i.d.\ samples $k_n$ in each row must be asymptotically bounded. We employ large deviations theory to prove the Gumbel attraction of Mahonian and Eulerian distributions. It is shown that for the two classes, different bounds on $k_n$ ensure this.
\end{abstract}

\maketitle

%
%

\section{Introduction}
Statistics on random permutations, such as the number of inversions, descents, length and number of cycles, have a long history and have also attracted recent interest \cite{bruck2019central, chatterjee2017central, kahle2020counting}. These statistics can often be generalized from permutations to elements of finite Coxeter groups, where the validity of results on the statistics may depend on group theoretic properties. This is visible, for example, for the large rank asymptotics of the underlying Coxeter groups, as central limit theorems do not hold if the variance of permutation statistics is too small, which can happen in dihedral groups.

Central limit theorems (CLTs) for inversions and descents on finite Coxeter groups were found by Kahle \& Stump \cite{kahle2020counting}. For the symmetric groups, a proof of the CLT for both statistics was already given by Bender~\cite{bender1973central}.  Chatterjee \& Diaconis~\cite{chatterjee2017central} used the method of interaction graphs developed in \cite{chatterjee2008new} to prove a CLT for the sum of descents and inverse descents.  Röttger \& Brück~\cite{bruck2019central,   rottger2018asymptotics} and F\'{e}ray \cite{feray2020central} extended this to other types of Coxeter groups. The work of Conger \& Viswanath~\cite{conger2007normal} provides CLTs for permutations on multisets, and He~\cite{he2022central} studied a CLT for the two-sided descent statistic on a Mallows-distributed permutation.

In this paper, we initiate the study of extreme values of permutation statistics. That is, we study the distribution of the maximum value of a permutation statistic over a collection of independent samples.  In extreme value theory, one seeks limit theorems in the way of $a_n^{-1}(M_n - b_n) \overset{\D}{\longrightarrow} G,$ where $M_n$ is the maximum of certain random variables, $G$ is a so-called \textit{extreme value distribution,} and $a_n, b_n$ are deterministic sequences. If $M_n := \max\{X_1, \ldots, X_n\}$ is based on a sequence $(X_n)_{n \in \N}$ of i.i.d.\ random variables, then $G$ is either a Gumbel, a Fr\'{e}chet, or a Weibull distribution, and $(X_n)_{n \in \N}$ is said to be in the \textit{max-domain of attraction} (MDA) of $G$. The Fr\'{e}chet distribution attracts heavy-tailed distributions, and the Weibull distribution attracts smooth distributions with a finite right endpoint, e.g., the uniform distribution. The Gumbel distribution is an intermediate case that attracts many important distributions, such as the normal and exponential distributions. See, e.g., \cite[Chapter 1]{leadbetter2012extremes} or \cite[Section 21.4]{van2000asymptotic} for comprehensive overviews.

If $(X_n)_{n \in \N}$ follows a finitely supported discrete distribution, then there is a finite right endpoint $x^*$ with $\PP(X_1 = x^*) > 0.$ Therefore, with probability $1$ there is some $N \in \N$ such that $M_N = M_{N+1} = \ldots = x^*$. Hence, no affine-linear rescaling can achieve a non-degenerate limit behavior of the maxima. This also affects permutation statistics if we consider them on a single Coxeter group.
%
This lack of non-degenerate extreme value limits for i.i.d.\ sequences also affects infinitely supported discrete distributions, e.g., the Poisson and geometric distributions. A necessary condition for the existence of a non-degenerate extreme value limit is given in \cite[Theorem 1.7.13]{leadbetter2012extremes}. Another interesting work on discrete distributions in MDAs is \cite{shimura2012discretization}. For the discrete distributions that are outside of all MDAs, it is possible to construct a row-wise independent triangular array $(X_{nj})_{j=1,\ldots, k_n}$ with row-wise maxima $M_n := \max\{X_{n1}, \ldots, X_{nk_n}\}$ in order to cover entire classes of distributions. This approach can also be used for permutation statistics over families of Coxeter groups.    

To this day, there is no complete classification of non-degenerate limits of $(M_n - b_n)/a_n$
for such triangular arrays. Several efforts have been made for triangular arrays consisting of common probability distribution families, with the limit often being the Gumbel distribution.  One common technique, which we also employ, is to draw connections to the i.i.d.\ extreme value behavior of the standard normal distribution. Anderson et al.~\cite{anderson1997maxima} proved a Gumbel limit for
a uniform triangular array $(X_{nj})_{j=1,\ldots,n}$ of Poisson variables $(R_{n,i}) \sim \mathrm{Po}(\lambda_n),$ where the sequence $(\lambda_n)_{n \in \N} \subseteq \N$ satisfies a minimum growth rate. Dkenge et al.~\cite{dkengne2016limiting} gave a characterization for the more general case of a row-wise stationary triangular array $(\xi_{nj})_{j=1,\ldots, k_n}$, using a suitable growth rate of $k_n$ and well-known mixing conditions as stated, e.g., in Leadbetter et al.~\cite[Section~3.2]{leadbetter2012extremes}. However, the framework of Dkenge et al.\ additionally requires that all $\xi_{nj}$ have an infinite right endpoint. Recently, Panov and Morozova~\cite{panov2021extreme} classified various mixture models with heavy-tailed impurity, including situations where the extreme value limit is neither of the Gumbel, Fr\'{e}chet or Weibull distributions.

Regarding extreme value theory for permutation statistics, Mladenovi\'{c}~\cite{mladenovi2002note} studied the extreme value behavior of the \textit{largest gap statistic} on the symmetric group. For a permutation $\omega = (a_1, \ldots, a_n) \in S_n$, there are the random variables $X_{nj}(\omega) := |a_j - a_{j+1}|$, $j=1, \dotsc, n$, $a_{n+1} = a_1$. These form a triangular array whose row-wise maximum $M_n := \max\{X_{n1}, \ldots, X_{nn}\}$ is the largest gap occurring in~$\omega$. Mladenovi\'{c} proved that the sequence $(M_n)_{n \in \N}$ is attracted to the Weibull-2-distribution. However, all entries of this triangular array are based on \textit{the same permutation} in each row, and thus are not stochastically independent.

For the number of inversions and descents, no such approach is feasible. Instead, we construct triangular arrays of \textit{independent samples} from the Coxeter groups in each row. This means that we consider sequences of Coxeter groups $(W_n)_{n \in \N}$ with ranks $n = \rk(W_n)$ and triangular arrays $(X_{nj})_{j=1,\ldots,k_n}$ of permutation statistics. On each Coxeter group $W_n$, we draw $k_n$ samples of the permutation statistics. To achieve a non-degenerate extreme value behavior, the sequence $(k_n)_{n \in \N}$ must be divergent. If $k_n$ grows only slowly, then the CLT suggests that the rows for which $k_n$ is large behave similarly to the standard normal distribution. If $k_n$ grows too fast, then the discrete character of the permutation statistics will be dominant. Therefore, the growth rate of $k_n$ is to be determined. In this work, we prove that under suitable conditions on the growth of $k_n$, the row-wise maximum $M_n := \max\{X_{n1}, \ldots, X_{nk_n}\}$ is attracted to the Gumbel distribution.

This paper is structured as follows. Section~\ref{sec:prelim} gives preliminaries about Coxeter group theory and permutation statistics, including well-known representations of generating functions for the numbers of inversions and descents. Section~\ref{sec:tail} reviews the main tool for our results on extremes of permutation statistics. Section~\ref{sec:results} gathers the main results.  The brief Section~\ref{sec:BerryEsseen} discusses the growth rate of $k_{n}$ under the Berry--Esseen assumption.

Following the common $O$-notation, we express magnitude relations for positive sequences $(a_n)_{n \in \N}, (b_n)_{n \in \N}$ as follows:
\begin{itemize}
\item $a_n = O(b_n)$ or $a_n \preccurlyeq b_n$ means that the sequence $a_n$ grows at most as fast as~$b_n$, i.e., $\limsup_{n\rightarrow\infty} \frac{a_n}{b_n} < \infty$. 
\item $a_n = o(b_n)$ means that $a_n$ grows slower than $b_n$, or is negligible compared to $b_n$, i.e., $\lim_{n\rightarrow\infty} \frac{a_n}{b_n} = 0$. This is also written as $a_n \prec b_n$ or $b_n \succ a_n$.
\item $a_n = \Theta(b_n)$ means that $a_n$ and $b_n$ have the same order of magnitude, i.e., both $a_n = O(b_n)$ and $b_n = O(a_n)$ hold.
\end{itemize}



%

\section{Preliminaries about Coxeter groups and permutation statistics}
\label{sec:prelim}

\begin{definition}
Let $S_n$ be the symmetric group on $\{1, \ldots, n\}$. For any $\pi \in S_n$, the number of \textit{inversions} and the number of \textit{descents} are
\begin{align*}
    \inv(\pi) &:= \#\{(i,j) \mid 1 \leq i < j \leq n, \pi(i) > \pi(j)\}, \\
    \des(\pi) &:= \#\{i \mid 1 \leq i < n, \pi(i) > \pi(i+1)\},
\end{align*}
where $\#$ denotes cardinality of a set. If we interpret $S_n$ as a probability space and draw each permutation with uniform probability $1/n!,$ these two quantities yield random variables that we denote by $X_\inv$ and $X_\des$. The probability distribution of $X_\inv$ is known as the \textit{Mahonian distribution} and that of $X_\des$ is known as the \textit{Eulerian distribution}. 
\end{definition}

A \textit{Coxeter group} $W$ is a group generated by a set $S$ whose
elements satisfy the following relations (and no further relations):
\begin{itemize}
\item $s^2 = e$ for all $s \in S$, where $e$ denotes the neutral
  element.
\item For any $s, s' \in S$, $s \not= s'$, there is a number
  $M(s,s') \in \{2, 3, \ldots, \infty\}$ such that
  $(ss')^{M(s,s')} = e$. In other words, $ss' \not= e$ and $M(s,s')$
  denotes the order of~$ss'$.  Here, $\infty$ means that no such
  relation holds.
\end{itemize}
The cardinality of $S$ is the \textit{rank} $\rk(W) = \#S$ of the Coxeter group $W$ and the pair $(W,S)$ is commonly called a \emph{Coxeter system}.

Inversions and descents can be defined on any Coxeter system $W = (W,S)$ using the word length function $l(w)$, which is the length of the shortest possible expression $w = s_1 \cdots s_k$ with generators $s_i \in S$. Let $T := \{wsw^{-1} \mid w \in W, s \in S\}$ be the set of \textit{reflections} of~$W$.
\begin{itemize}
\item The (right) \textit{inversions} of $w \in W$ are the set $\{t \in T\!\!: l(wt) < l(w)\}$.
\item The (right) \textit{descents} of $w \in W$ are the set $\{s \in S\!\!: l(ws) < l(w)\}$.
\end{itemize}
Again, the cardinalities of these sets are $\inv(w)$ and $\des(w)$, and the numbers of inversions and descents of a random Coxeter group element are $X_\inv$ and~$X_\des$, respectively.

\begin{remark} \label{rem:classification} 
A Coxeter group is \textit{irreducible} if it is not a direct product of smaller Coxeter groups. There is a complete classification of finite and irreducible Coxeter groups, see~\cite{coxeter1935complete}. There are three important families of finite irreducible Coxeter groups with increasing rank:
\begin{itemize}
\item the symmetric groups on $\{1, \ldots, n\}$ that are denoted by $A_{n-1}$ instead of $S_n$, because $A_{n-1} = S_n$ is generated by the $n-1$ adjacent transpositions $\tau_i$. Indeed, these satisfy $\tau_i^2 = e$ for $i = 1, \dotsc, n-1$,   $(\tau_i \tau_{i+1} \tau_i)^3 = e$ for $i = 1, \ldots, n-2$ and $(\tau_i\tau_j)^2 = e,$ if $|i-j| \geq 2$.
\item the groups $B_n$ of \textit{signed permutations} consisting of bijective maps $\pi\colon \{\pm 1, \dotsc, \linebreak \pm n\} \rightarrow \{\pm 1, \dotsc, \pm n\}$ that satisfy $\pi(-i) = -\pi(i)$ for $i = 1, \dotsc, n$. In this model, each element of $B_{n}$ consists of a permutation and a sign for each $i=1,\dotsc, n$, hence the name. The group is generated by adjacent transpositions of $A_{n-1}$ together with the map that inverts the sign of the first element, i.e., $(1, \ldots, n) \mapsto (-1, 2, \ldots, n)$.
\item the subgroups $D_n \subseteq B_n$ of \textit{even-signed permutations} consisting of elements with an even number of negative signs. This group is generated by the adjacent transpositions of $A_{n-1}$ together with the map that inverts the sign of the first two elements, i.e., $(1, \ldots, n) \mapsto (-1, -2, 3, \ldots, n)$.
\end{itemize}
There are also the dihedral groups $I_2(m), m \in \N,$ known as the isometry groups of regular $m$-gons. As all of these groups have rank 2, they are not a family of increasing rank. The rank grows only if we build direct products of dihedral groups. Finally, there are eight exceptional groups that are not addressed in this paper.
\end{remark}

Numerical information about inversions and descents on a Coxeter group~$(W,S)$ is stored in the \textit{generating functions} 
\begin{align*} \G_\inv(W;z) &:= \sum_{w \in W} z^{\inv(w)}, & 
    \G_{\des}(W;z) &:= \sum_{w \in W} z^{\des(w)}. 
\end{align*}
Obviously, these are polynomials with integer coefficients. The generating function of descents is known as the \textit{Eulerian polynomial}.

The following explicit formula for $\G_{\inv}$ is found in \cite[Chapter~7]{bjorner2006combinatorics}. The degrees appearing in Theorem~\ref{thm:factorGF} are certain integers associated with $W$ in the context of invariant theory.  For our purposes, it is sufficient to know that they exist, are known for all irreducible finite Coxeter groups, and are easy to derive for products.

\begin{theorem} \label{thm:factorGF}
Let $W$ be a finite Coxeter group with $\rk(W) = n$. Then, 
\[\G_\inv(W;z) = \prod_{i=1}^n (1 + z + \ldots + z^{d_i - 1}),\]
where $d_1, \ldots, d_n$ are the degrees of~$W$.
\end{theorem}

The generating function $\G_{\des}$ also has a decomposition, even into linear factors. This was proved by Brenti~\cite{brenti1994q} for all irreducible types except~$D$, and that case was resolved by Savage \& Visontai~\cite{savage2015s}. They proved that the Eulerian polynomial of these groups is real-rooted. From this, it is trivial to conclude that the roots are negative, since all coefficients of the Eulerian polynomial are positive. Apart from the sign, not much is known about the roots.

\begin{theorem} 
Let $W$ be a finite Coxeter group with $\rk(W) = n$. Then, 
\[\G_{\des}(W;z) = \prod_{i=1}^n (z+q_i)\]
for some $q_1, \ldots, q_n > 0$.
\end{theorem}

Whenever a generating function factors, the corresponding statistic is a sum of independent contributions corresponding to the factors. Therefore, $X_\inv$ and $X_\des$ can be written as sums of independent (but not identically distributed) variables, which prove to be key for the extreme values of these statistics.

\begin{corollary} \label{cor:sumDec}
Let $W$ be a finite Coxeter group with $\rk(W) = n$. Then:
\begin{enumerate}
\item[a)] $X_\inv = \sum_{i=1}^n X_\inv^{(i)},$ where $X_\inv^{(i)} \sim U\left\{0, 1, \ldots, d_i - 1\right\}$ and  $d_1, \ldots, d_n$ are the degrees of $W$.
\item[b)] $X_\des = \sum_{i=1}^n X_\des^{(i)},$ where $X_\des^{(i)} \sim \Bin\left(1, (1 + q_i)^{-1}\right)$ and $q_1, \ldots, q_n$ are the negatives of the zeroes of $\G_\des(W)$.
\end{enumerate}
\end{corollary}

\begin{remark} \label{rem:orders}
The means and variances of $X_\inv$ and $X_\des$ have been listed by Kahle \&  Stump~\cite[Corollary~3.2~and~4.2]{kahle2020counting}. The essential magnitudes are as follows. Whenever $W \in \{A_n\}_{n\in \N}, \{B_n\}_{n\in \N}, \{D_n\}_{n\in \N},$ then
\begin{align*}
    \E(X_\inv) &= \Theta(n^2), & \E(X_\des) &= n/2, \\ 
    \Var(X_\inv) &= \Theta(n^3), & \Var(X_\des) &= \Theta(n).
\end{align*}
\end{remark}

%
%

\section{Tail equivalence for non-i.d. sums}
\label{sec:tail}
In what follows, let $\Phi(x) = \int_{-\infty}^x \frac{1}{\sqrt{2\pi}}e^{-x^2/2}\dx$ be the cumulative distribution function (CDF) of the standard normal distribution. It is well known (e.g., \cite[Thm.~1.5.3]{leadbetter2012extremes}) that i.i.d.\ standard normal variables $N_1, N_2, \ldots$ are attracted to the Gumbel distribution $\Lambda(x) = \exp(-\exp(-x))$ by virtue of
\[
  \frac{M_n - \beta_n}{\alpha_n} \overset{\D}{\longrightarrow}
  \Lambda, \qquad M_n := \max\{N_1, \ldots, N_n\},
\]
using the constants $\alpha_n = (2\log n)^{-1/2}$,
$\beta_n = \sqrt{2\log(n)} - \alpha_n\bigl(\log \log n +
\log(4\pi)\bigr)/2$. If a family $F_1, F_2, \ldots$ of standardized distributions has \textit{tail equivalence} in the sense that
\begin{equation} 
1 - F_n\left(x_n\right) \sim 1 - \Phi(x_n)\quad \Longleftrightarrow \quad\frac{1 - F_n(x_{n})}{1 - \Phi(x_{n})} = 1 + o(1), \label{eqn1}
\end{equation}
where $(x_n)_{n \in \N} \subseteq \R$ is a sequence that is commonly limited in growth, then we can use that for any fixed $x$,
\[
  n\bigl(1 - \Phi(\alpha_n x + \beta_n)\bigr) \longrightarrow e^{-x},
\]
and we plug $x_n = \alpha_n x + \beta_n$ in
\eqref{eqn1}. Then, if the sequence $(\alpha_n x + \beta_n)_{n \in \N}$ does not violate the conditions for $x_n$ required in \eqref{eqn1}, combining both limit processes yields
\[
  n\bigl(1 - F_n\left(\alpha_n x + \beta_n\right)\bigr) \sim
  n\bigl(1 - \Phi(\alpha_n x + \beta_n)\bigr) \longrightarrow e^{-x}. \qedhere
\]
The subject of tail equivalence is closely related to the field of \textit{large deviations theory.} Based on limit theorems such as the strong law of large numbers or the CLT, this theory deals with bounds and asymptotic quantifications for the probabilities of large deviations from the limit. See \cite[Chapters 1 and~2]{dembo2009large} for an introduction. We assume the following framework for all theorems in this section.



\subsection*{Framework}
Let $X_1, X_2, \ldots$ be an at most countable sequence of independent (not necessarily i.d.) random variables. Without loss of generality, assume that all $X_k$ are centered. For $n \in \N,$ let $S_n := X_1 + \ldots + X_n$. Moreover, let $\sigma_k^2 = \E(X_k^2)$ for all $k = 1, \ldots, n$, let $s_n^2 := \sigma_1^2 + \ldots + \sigma_n^2$, and let $F_n$ denote the CDF of $S_n/s_n$. We aim to demonstrate tail equivalence between $F_n$ and $\Phi$ as described in \eqref{eqn1}. Upon additionally assuming that $X_1, X_2, \ldots$ are identically distributed, a seminal result on large deviations of $(1 - F_n)/(1 - \Phi)$ is due to Cram\'{e}r \cite{cramer1938nouveau}.

\begin{theorem}[cf.~Cram\'{e}r \cite{cramer1938nouveau}] \label{thm3.2.1}
Under the above framework, assume that $X_1, X_2, \ldots$ are i.i.d. and that the moment generating function of $X_1$ exists in a neighborhood of the origin. If $x = o(\sqrt{n}),$ then
\[
    \frac{1 - F_n(x)}{1 - \Phi(x)} = \exp\left(\frac{x^3}{\sqrt{n}}\mathcal{L}\left(\frac{x}{\sqrt{n}}\right)\right)\bigl(1 + o(1)\bigr)\,,
\]
where $\mathcal{L}(x) = \displaystyle{\sum\nolimits_{k=0}^\infty a_kx^k}$ is a power series with coefficients depending on the cumulants of~$X_1$.
\end{theorem}
%
%
A similar theorem that omits the assumption of identical distribution was developed by Feller \cite{feller1943generalization}. This theorem imposes boundedness assumptions on the random variables, therefore it is not a generalization of Theorem~\ref{thm3.2.1}.

\begin{theorem}[cf.~Feller \cite{feller1943generalization}] \label{thmfeller2}
Let $(\lambda_n)_{n \in \N}$ be a  sequence of constants such that $\lambda_n \longrightarrow 0$ and
\begin{equation}
    \forall k = 1, \ldots, n \negthickspace: |X_k| < \lambda_n s_n\,. \label{3.1}
\end{equation}
Let $x > 0$ be fixed and assume that
\[
    \forall n \in \N \negmedspace: 0 < \lambda_nx < (3 - \sqrt{5})/4 \approx 0.19.
\]
Then, there is a constant $\vartheta$ and a power series $Q_n(x) = \displaystyle{\sum\nolimits_{\nu=1}^\infty    q_{n,\nu}x^\nu}$ with coefficients $q_{n,\nu}$ depending on the first $\nu+2$ moments of $X_n$ such that
\begin{equation}
    1 - F_n(xs_n) = \exp\left(-\frac{1}{2}x^2Q_n(x)\right)\left(1 - \Phi(x) + \vartheta \lambda_ne^{-x^2/2}\right). \nonumber
\end{equation}
If, in particular, $0 < \lambda_nx < 1/12,$ then 
$|q_{n,\nu}| < \frac{1}{7}(12\lambda_n)^\nu$.
\end{theorem}

\begin{remark} \label{crucialremark} Theorem~\ref{thmfeller2} concerns finite sequences of random variables and does not contain any asymptotic statement. Nevertheless, it can be applied for each $n \in \N$ on a uniform triangular array $(X_{nj})_{j=1,\ldots,n}$ to draw asymptotic conclusions. As stated by Feller~\cite{feller1943generalization}, if it is possible to choose a sequence $(\lambda_n)$ with $\lambda_n = O(n^{-1/2})$, and if $x = x_n = o(n^{1/6}),$  then
\begin{align*}
    Q_n(x) = q_{n,1}x + \sum_{\nu=2}^\infty q_{n,\nu}x^\nu \leq~& \frac{12}{7}\lambda_nx + O(n^{-2/3}) = O(n^{-1/3}) \\
    {\Longrightarrow}~&\exp\left(-\frac{x^2}{2} Q_n(x)\right) \longrightarrow 1.
\end{align*}
Furthermore, $e^{-x^2/2}$ is bounded by $1$ and $\vartheta$ is a constant, so $\vartheta \lambda_n e^{-x^2/2} \longrightarrow 0$. Thus, whenever the aforementioned conditions are satisfied, we have $1 - F_n(x) \sim 1 - \Phi(x)$, the desired tail equivalence.

In comparison, Cram\'{e}r's Theorem~\ref{thm3.2.1} allows for the broader regime $x = o(n^{1/2})$, and it has also been generalized to non-i.d. independent random variables. We now introduce a large deviations theorem of Petrov \& Robinson \cite{petrov2008large}, which is, to the best of our knowledge, the weakest known generalization of Theorem~\ref{thm3.2.1}.

Under the above framework, let $L_j$ be the cumulant generating function of $X_j,$ that is, $L_j(z) = \log\left(\E\left(e^{zX_j}\right)\right)$. We assume that for some $H > 0$, all functions $L_j$ are analytic within the circle $\{z \in \C \negmedspace: |z| < H\}$. Moreover, we assume the existence of constants $(c_j)_{j \in \N}$ such that $\forall |z| < H, j \in \N \negmedspace: |L_j(z)| < c_j$ and
\begin{equation}
  \limsup_{n\rightarrow\infty} \sum_{j=1}^n \frac{c_j}{n} < \infty\,. \label{3.2}
\end{equation}
At last, we require that the variances $s_n^2$ grow at least linearly,
that is,
\begin{equation}
  \liminf_{n\rightarrow\infty} \frac{s_n^2}{n} > 0\,. \label{3.3}
\end{equation}
\end{remark}

\begin{theorem}[see Petrov \& Robinson \cite{petrov2008large}, Theorem 2.1] \label{thm3.2.4}
Given the conditions \eqref{3.2} and \eqref{3.3}, it holds that for $x = o(\sqrt{n}),$
\[
    {\frac{1 - F_n(x)}{1 - \Phi(x)} = \exp\left(\frac{x^3}{\sqrt{n}}\mathcal{L}_n\left(\frac{x}{\sqrt{n}}\right)\right)\bigl(1 + o(1)\bigr)\,,}
\]
where 
$\mathcal{L}_n(x) = \displaystyle{\sum\nolimits_{k=0}^\infty
  a_{kn}x^k}$ is a power series with coefficients $a_{kn}$ expressed
in terms of the cumulants of $X_1, \ldots, X_n$ of order up to and
including $n+3$.
\end{theorem}

\begin{remark} \label{bem3.2.5} For the extended regime 
  $n^{1/6} \prec x \prec n^{1/2}$, it is not trivial to obtain tail
  equivalence from Theorems~\ref{thm3.2.1} and~\ref{thm3.2.4}. To do so, we additionally need to demonstrate
\begin{align*}
    {\exp\left(\frac{x^3}{\sqrt{n}}\mathcal{L}_n\left(\frac{x}{\sqrt{n}}\right)\right)} & = 1 + o(1) \\
    {\Longleftrightarrow \frac{x^3}{\sqrt{n}}\mathcal{L}_n\left(\frac{x}{\sqrt{n}}\right)} & = o(1)\,.
\end{align*}
{The term $x^3/\sqrt{n}$ can become as large as $o(n)$. It is controlled only if $x = o(n^{1/6}),$ which is the same regime as in Theorem~\ref{thmfeller2}. For broader regimes, we need to control the power series $\mathcal{L}_n$. For $j, k \in \N,$ let $\gamma_{kj}$ be the $k$-th cumulant of $X_j$ and let} 
\[
    \Gamma_{kn} = \sum_{i=1}^n \frac{\gamma_{ki}}{n}\,.
\]
According to \cite[p.\ 2985]{petrov2008large}, the first coefficient of $\mathcal{L}_n$ is
\[
    a_{0,n} = \frac{\Gamma_{3,n}}{6\Gamma_{2,n}^{3/2}}\,.
\]
If $a_{0,n}$ is non-zero, then it is impossible to control $\mathcal{L}_n(x/\sqrt{n})$ for any $n^{1/6} \prec x \prec n^{1/2}$. To obtain tail equivalence from Theorem~\ref{thm3.2.4} within the extended regime $n^{1/6} \prec x \prec n^{1/2}$, it is necessary that $a_{0,n} = o(n^{-1})$. In the intermediate case of $a_{0,n} = o(1)$ and $a_{0,n} = \Omega(n^{-1}),$ the regime of $x$ can be extended at least partially. In that case, further coefficients of $\mathcal{L}_n$ may have to be taken into account.
\end{remark}


%
%

\section{Results for extremes of permutation statistics}
\label{sec:results}

Using the techniques discussed in the previous section, we can derive the extreme value behavior of triangular arrays on sequences of Coxeter groups. Let $(k_n)_{n \in \N}$ be a divergent sequence of natural numbers. We consider a triangular array where in the $n$-th row, we have $k_n$ i.i.d.~samples $X_{n1}, \ldots, X_{nk_n}$ with $X_{n1}$ being the number of inversions or descents on some finite Coxeter group of rank $n$. We suppose that the triangular array contains only samples of either $X_\inv$ or $X_\des,$ but not both.

It is important to distinguish whether dihedral groups are involved or not. For simplicity, we refer to finite irreducible Coxeter groups of $A$-, $B$- or $D$-type as the \textit{classical Weyl groups} since they are the Weyl groups of the classical groups. 

\subsection{Sequences of classical Weyl groups} 
We consider a sequence of classical Weyl groups $(W_n)_{n \in \N}$ with $\rk(W_n) = n~\forall n \in \N$. Let $X_\inv^{(n)}, X_\des^{(n)}$ be the number of inversions and descents on~$W_n$, respectively. With Corollary~\ref{cor:sumDec} we write
\begin{equation*}
  X_\inv^{(n)} = \sum_{i=1}^{n} X_\inv^{(n,i)} 
  \qquad \text{and}
  \qquad 
  X_\des^{(n)} = \sum_{i=1}^{n} X_\des^{(n,i)},
\end{equation*}
where $X_\inv^{(n,i)} \sim U\left\{0, 1, \ldots, d_i^{(n)} - 1\right\}$ and $X_\des^{(n,i)} \sim \Bin\left(1, \left(1 + q_i^{(n)}\right)^{-1}\right)$.  


\begin{remark} \label{bem4.2}
Since Theorem~\ref{thm3.2.4} permits a broader regime of $x$ than Theorem~\ref{thmfeller2}, it is preferable to apply Theorem~\ref{thm3.2.4} for both $X_\inv$ and $X_\des$. However, it turns out that the conditions of Theorem~\ref{thm3.2.4} are not satisfied for $X_\inv$. For $X_\inv^{(n,i)} \sim U\left\{0, 1, \ldots, d_i^{(n)} - 1\right\}$, the cumulant generating function is
\[
    L_i(z) = \log\left(\frac{1}{d_i^{(n)}}\sum_{k=0}^{d_i^{(n)}-1} e^{zk}\right) = \log\left(\frac{1 - e^{d_i^{(n)}z}}{d_i^{(n)}(1 - e^z)}\right).
\]
For some $H > 0,$ we need to find $c_i$ such that $L_i(z) < c_i$ $\forall |z| < H.$ In particular,
\[
    c_i \geq L_i(H) = \log\left(\frac{1 - (e^H)^{d_i^{(n)}}}{d_i^{(n)}(1 - e^H)}\right).
\]
Due to $e^H > 1$, we have that $L_i(H)$ grows linearly in $i$, as its argument grows exponentially in $i$. Therefore, $\displaystyle{\sum\nolimits_{j=1}^n c_j/n}$ grows linearly as well and is not bounded, so condition \eqref{3.2} is violated. In the case of descents, condition~\eqref{3.2} is not violated. However, we need to examine the power series $\mathcal{L}_n$ in order to determine the appropriate regime of $x$. The second, third, and fourth cumulants of $X_\des^{(n,i)} \sim \Bin\left(1, \bigl(1 + q_i^{(n)}\bigr)^{-1}\right) =: \Bin(1, p_i)$ are
\begin{align*}
    {\gamma_{2,i}} & {~= p_i(1-p_i),} \\
    {\gamma_{3,i}} & {~= p_i(1-p_i)(1 - 2p_i),} \\
    {\gamma_{4,i}} & {~= p_i(1-p_i)(1 - 6\gamma_{2,i}).}
\end{align*} 
Recall that $a_{0,n} = \Gamma_{3,n}/\Gamma_{2,n}^{3/2}$. The third cumulant $\gamma_{3,i}$ equals the third central moment of $X_{\des}^{(n,i)}$. Therefore, the sum $\sum\nolimits_{i=1}^n \gamma_{3,i}$ equals the third central moment of $X_{\des},$ which is zero as the distribution of $X_{\des}$ is symmetric for all finite Coxeter groups. In conclusion, we have $a_{0,n} = 0$. However, we need to take the second coefficient of $\mathcal{L}_n$ into account, which is, according to \cite[p.\ 2985]{petrov2008large}:
\[
    a_{1,n} = \frac{\Gamma_{4,n}\Gamma_{2,n} - 3\Gamma_{3,n}^2}{24\Gamma_{2n}^3} = \frac{\Gamma_{4,n}\Gamma_{2,n}}{24\Gamma_{2,n}^3}.
\]
Due to $\gamma_{4,i} = \gamma_{2,i}(1 - 6\gamma_{2,i})$ and $1 - 6\gamma_{2,i}
\in [-1/2, 1)$ for all $p_i \in (0,1)$, we have $|\gamma_{4,i}| < |\gamma_{2,i}|
\forall i = 1, \ldots, n \Longrightarrow |\Gamma_{4,n}| < |\Gamma_{2,n}|,$
giving $a_{1,n} \leq \Gamma_{2,n}^{-1/2}/24$. However, due to $\Gamma_{2,n} =
n^{-1}\Var(X_\des) = \Theta(1)$, 
this only implies $a_{1,n} = O(1)$.\footnote{The published version states
$a_{1,n} = O(1/n)$ which is not correct.}
In light of Remark~\ref{bem3.2.5}, we can extend the regime of $x$ to $x = o(n^{1/4})$ to ensure tail equivalence for $X_{\des}$. 
\end{remark}

\noindent Since Theorem~\ref{thm3.2.4} cannot be applied to inversions, we need
to use Theorem~\ref{thmfeller2} to achieve tail equivalence. Indeed, this is
successful because the components $X_\inv^{(n,i)}$ are bounded and the variance
of $X_\inv$ is of appropriate magnitude. This argument also works for descents,
but for these, we can use the broader regime $x = o(n^{1/3})$ according to
Remark~\ref{bem4.2}. We summarize these observations for the numbers of
inversions and descents on classical Weyl groups as follows.\footnote{The
published version of this paper contained the too optimistic 
requirement $k_n = \exp(o(n^{2/3}))$ in assumption (b) in the following theorem.
Via $a_{1,n} = O(1)$ in Remark~\ref{bem4.2}, only $o(n^{2/4})$ is possible.}

\begin{theorem} \label{thm:mainWeyl} 
Let $(W_n)_{n \in \N}$ be a sequence of classical Weyl groups with $\rk(W_n) = n$ for all $n \in \N$.  Let $(X_{nj})_{j=1,\ldots,k_n}$ be a row-wise i.i.d.\ triangular array with either $X_{n1} \overset{\D}{=} X_\inv$ $\forall n \in \N$ or $X_{n1} \overset{\D}{=} X_\des$ $\forall n \in \N,$ where:
\begin{enumerate*}
    \item[\emph{(a)}] If $X_{n1} \overset{\D}{=} X_\inv$ $\forall n \in \N,$ then we assume $k_n = \exp\bigl(o(n^{1/3})\bigr).$
    \item[\emph{(b)}] If $X_{n1} \overset{\D}{=} X_\des$ $\forall n \in \N,$ then we assume $k_n = \exp\bigl(o(n^{1/2})\bigr).$
\end{enumerate*}
Let $M_n := \max\{X_{n1}, \ldots, X_{nk_n}\}$. Let $\mu_n := \E(X_{n1}),$ $s_n^2 := \Var(X_{n1}),$ and  
\begin{align*}
    \alpha_n &= \frac{1}{\sqrt{2\log k_n}}\,, & \beta_n &= \frac{1}{\alpha_n} - \frac{1}{2}\alpha_n\bigl(\log \log k_n + \log(4\pi)\bigr)\,.
\end{align*}
Put $a_n := \alpha_n s_n$ and $b_n := \beta_ns_n + \mu_n$. Then, for all $x \in \R$ we have
\[\PP(M_n \leq a_nx + b_n) \longrightarrow \exp\bigl(-\exp(-x)\bigr)\,.\]
\end{theorem}

\begin{proof}
  Let $F_n$ be the CDF of $X_{n1}$. Each $F_n$ is a sum of $n$ summands by Corollary~\ref{cor:sumDec}. In the case of $(X_{nj})_{j=1,\ldots,k_n}$ being numbers of inversions, applying Theorem~\ref{thmfeller2} separately for each $n \in \N$ gives
\[
    1 - F_n(xs_n) = \exp\left(-\frac{1}{2}x^2Q_n(x)\right)\left(1 - \Phi(x) + \vartheta \lambda_ne^{-x^2/2}\right), \quad n = 1, 2, \ldots.
\]
The condition $\lambda_n = O(n^{-1/2})$ can be equivalently expressed as $|X_k| = O(n^{-1/2}s_n)$. The degrees of finite Coxeter groups are bounded by $2n$, and the values of the centered variables $X_\inv^{(n,i)} - \E\left(X_\inv^{(n,i)}\right)$ are bounded by $n$. Furthermore, $s_n = O(n^{3/2})$ holds.
Therefore, the choice of $\lambda_n = O(n^{-1/2})$ is possible. Upon undoing the centering assumed in Theorem~\ref{thmfeller2}, we obtain according to Remark~\ref{crucialremark}:
\[ 
    1 - F_n(\mu_n + s_n y) \sim 1 - \Phi(y), \quad
    \text{whenever~} y = o(n^{1/6}).
\]
Plugging in $y = \alpha_n x + \beta_n$ and treating $x$ as a constant, the condition $\alpha_n x + \beta_n \prec n^{1/6}$ in Feller's theorem is satisfied due to $n \succ \log(k_n)^3$ by assumption (a).   In the case of $(X_{nj})_{j=1,\ldots,k_n}$ being numbers of descents, Theorem~\ref{thm3.2.4} and Remark~\ref{bem4.2} give $1 - F_n(\mu_n + s_n y) \sim 1 - \Phi(y)$ for $y = o(n^{1/4})$, which is satisfied for $y = \alpha_nx + \beta_n$ by assumption (b). Hence,
\[  k_n\bigl(1 - F_n(a_nx + b_n)\bigr) =
    k_n\Bigl(1 - F_n\bigl(\mu_n + s_n(\alpha_n x + \beta_n)\bigr)\Bigr) \longrightarrow e^{-x},
\]
proving the Gumbel attraction of the row-wise maxima~$M_n$ in both cases.
\end{proof}

\begin{remark}
The proof of Theorem \ref{thm:mainWeyl} fails when we try to further extend the regime of $k_n$. According to~\cite{feller1943generalization}, if $\lambda_n = O(n^{-1/2})$ and if $x$ is chosen in a way that $n^{1/6} \prec x \prec n^{1/4}$ in Theorem~\ref{thmfeller2}, then we have
\begin{equation}
    1 - F_n(xs_n) \sim \exp\left(-\frac{1}{2}q_{n,1}x^3\bigl(1 - \Phi(x)\bigr)\right), \label{eqn4.1}
\end{equation}
as $Q_n(x) = q_{n,1}x + \sum_{\nu=2}^\infty q_{n,\nu}x^\nu$ with $q_{n,1} = o(n^{-1/2}).$ However, $n^{1/2} \prec x^3 \prec n^{3/4},$ giving 
\[\exp\left(-\frac{1}{2}x^2Q_n(x)\right) = \exp\left(-\frac{1}{2}q_{n,1}x^3 + o(1)\right),\]
from which \eqref{eqn4.1} follows. The first coefficient $q_{n,1}$ is explicitly stated by Feller~\cite[Eq. (2.18)]{feller1943generalization} as
\[q_{n,1} = \frac{1}{3s_n^3}\sum_{i=1}^n \E\left(X_{ni}^3\right).\]
Considering the number of inversions on classical Weyl groups, we have $s_n^3 = \Theta(n^{9/2})$ and $X_{ni} \sim U(\{0, 1, \ldots, d_i - 1\})$. For a discrete uniformly distributed random variable, the third moment is
\[\E\left(X_{ni}^3\right) = \sum_{j=0}^{d_i} \frac{1}{d_i+1}j^3 = \frac{1}{d_i+1}\frac{d_i^2(d_i+1)^2}{4} = \frac{d_i^2(d_i+1)}{4} = \Theta(d_i^3)\,.\]
As the degrees of the classical Weyl groups are evenly spread across $2, \ldots, n+1$ or $2, 4, \ldots, 2n,$ respectively, we conclude that
\[\sum_{k=1}^n \E\left(X_{ni}^3\right) = \Theta(n^4) \Longrightarrow q_{n,1} = \Theta(n^{-1/2})\,.\]
In order to eliminate $-(1/2)q_{n,1}x^3$ in \eqref{eqn4.1}, we need $x^3 = o(n^{1/2}) \Longrightarrow x = o(n^{1/6}),$ which contradicts the assumption of $x \succ n^{1/6}$.
\end{remark}


\subsection{Arbitrary finite Coxeter groups}
The EVLT for $X_\des$ is based only on the application of Theorem~\ref{thm3.2.4}
to the representation of $X_\des$ in Corollary~\ref{cor:sumDec}b).  These
arguments require that $\Var(X_\des)$ grows linearly with respect to the rank,
which is particularly the case for products of classical Weyl groups. 
\footnote{The published version of this paper did not 
state the assumption $\Var(X_\des) = \Theta(n)$ explicitly.}
Therefore, we can state:

\begin{theorem} \label{thm4.5}
Let $(W_n)_{n \in \N}$ be a sequence of finite Coxeter groups with $\rk(W_n) = n~\forall n \in \N$, which satisfies $\Var(X_\des) = \Theta(n)$. Let $k_n = \exp\bigl(o(n^{1/2})\bigr),$ let $(X_{nj})_{j=1,\ldots,k_n}$ be a row-wise i.i.d.\ triangular array with $X_{n1} \overset{\D}{=} X_\des$ and let $M_n := \max\{X_{n1}, \ldots, X_{nk_n}\}$. Let $a_n, b_n$ be as in Theorem~\ref{thm:mainWeyl}. Then,
\[
    \PP(M_n \leq a_nx + b_n) \longrightarrow \exp\bigl(-\exp(-x)\bigr) \quad \forall x \in \R\,.
\]
\end{theorem}

The EVLT for inversions is based on Theorem~\ref{thmfeller2}. For arbitrary finite Coxeter groups, the condition $|X_{k}| = O(n^{-1/2}s_n)$ in the proof of Theorem~\ref{thmfeller2} is not trivially satisfied. For inversions, the $X_k = X_\inv^{(n,i)} - \E\left(X_\inv^{(n,i)}\right)$ can be bounded by the maximum degree $d_{\max}$ of the $n$-th Coxeter group~$W_n$. Therefore, this condition is written more descriptively as
\begin{equation}
    d_{\max} \preccurlyeq \frac{s_n}{\sqrt{n}}. \label{eq:growth1}
\end{equation}
Using the method of Theorem~\ref{thm:mainWeyl}, we can state a general EVLT for $X_\inv$ on sequences of finite Coxeter groups. Together with Theorem~\ref{thm4.5}, this is our main result as it gives sufficient conditions for the Gumbel attraction of $X_\inv$ and~$X_\des$.

\begin{theorem}\label{thm:mainArbitraryCoxeter}
Let $(W_n)_{n \in \N}$ be any sequence of finite Coxeter groups with $n = \rk(W_n)$. Let $k_n = \exp\bigl(o(n^{1/3})\bigr),$ let $(X_{nj})_{j=1,\ldots,k_n}$ be a row-wise i.i.d.\ triangular array with $X_{n1} \overset{\D}{=} X_\inv$ and let $M_n := \max\{X_{n1}, \ldots, X_{nk_n}\}$. Let $a_n, b_n$ be as in Theorem~\ref{thm:mainWeyl}. If condition \eqref{eq:growth1} holds, then 
  \[
    \PP(M_n \leq a_nx + b_n) \longrightarrow \exp\bigl(-\exp(-x)\bigr).
  \]
\end{theorem}

In the following subsections, we rephrase condition~\eqref{eq:growth1} more descriptively for certain products of finite irreducible Coxeter groups.

\subsection{Sequences of products of classical Weyl groups.}
Let $W_n = \prod_{i=1}^{l_n} W_{n,i},$ where each component $W_{n,i}$ is a classical Weyl group, and let $n = \rk(W_{n,1}) + \ldots + \rk(W_{n,l_n})$ denote the total rank. Then,
\[\Var(X_\inv^{W_n}) = \sum_{i=1}^{l_n} \Var(X_{\inv}^{W_{n,i}}).\]
For each $n$ and $i$, we have $\Var(X_{\inv}^{W_{n,i}}) = \Theta(\rk(W_{n,i}))$. However, the total variance $\Var(X_\inv^{W_n})$ is not of cubic order with respect to~$n$. By Corollary~\ref{cor:sumDec}a), $\Var(X_\inv^{W_n})$ still has an independent sum representation of $n$ summands. The maximum degree $d_{\max} \leq 2\max\{\rk(W_{n,1}), \ldots, \rk(W_{n,l_n})\}$ bounds these summands. Therefore, omitting the factor $2$ without asymptotic consequences, condition~\eqref{eq:growth1} can be written as
\begin{equation}\label{eq:growth2}
    d_{\max} \preccurlyeq \frac{1}{\sqrt{n}}\sqrt{\rk(W_{n,1})^3 + \ldots + \rk(W_{n,l_n})^3}.
\end{equation}

\begin{theorem}
  Let $W_n = \prod_{i=1}^{l_n} W_{n,i}$ be a sequence of direct
  products of classical Weyl groups, and let
  $(X_{nj})_{j=1,\ldots,k_n}$ be a row-wise i.i.d.\ triangular array
  with $X_{n1} \overset{\D}{=} X_\inv$. Let $k_n, M_n,$ $a_n, b_n$ be
  as in Theorem~\ref{thm:mainWeyl}. If condition \eqref{eq:growth2}
  holds, then for all $x \in \R \negmedspace:$
  \[
    \PP(M_n \leq a_nx + b_n) \longrightarrow \exp\bigl(-\exp(-x)\bigr).
  \] 
\end{theorem}

  As noted above, $X_\des$ satisfies the EVLT on all products of classical Weyl groups.

\subsection{Sequences involving dihedral groups} \label{dihedral}
In the following, we consider sequences of finite Coxeter groups consisting of dihedral components and classical Weyl group components. Here, it is more convenient to drop the convention $\rk(W_n) = n$.   Some care can be necessary when applying Theorem~\ref{thm:mainWeyl}.

\begin{example}  We consider a sequence of products of dihedral groups. Such a sequence consists of groups of even rank. We write 
\begin{equation}
\label{eq:dihedralprod}
 W_{n} = \prod_{i=1}^{h_n} I_2(m_{n,i})  
\end{equation}
for some $(m_{n,i})_{n \in \N, i=1, \ldots, n}$ and a growing sequence $(h_n)_{n \in \N}$. Then, $\rk(W_n) = 2h_n$. Now, the condition for applying Theorem~\ref{thm:mainWeyl} is $h_n \succ \log(k_n)^3$.
\end{example}

\begin{remark} \label{remarkdihedral} It has been stated by Kahle \& Stump~\cite{kahle2020counting} that for products of dihedral groups,
\begin{align*}
    \Var(X_\inv) &= \sum_{i=1}^{h_n} \frac{m_{n,i}^2+2}{12}, & \Var(X_\des) & = \sum_{i=1}^{h_n} \frac{1}{m_{n,i}}\,.
\end{align*}
Further, $I_2(m_{n,i})$ has degrees $2, m_{n,i}$. Therefore, the degrees of $W_n$ are $2, \ldots, 2$, $m_{n,1}, \ldots, m_{n,h_n}$ with $h_n$ twos. These formulas are now used to rephrase the condition \eqref{eq:growth1} for mixed products of dihedral groups and classical Weyl groups.
\end{remark}

Let $(W_n)_{n \in \N}$ be a sequence of finite Coxeter groups and
write $W_n = G_n \times I_n$, where $G_n$ contains only classical
components and $I_n$ contains only dihedral components as
in~\eqref{eq:dihedralprod}.  We use the following additional notation:
\begin{align*}
    r_n &:= \rk(G_n)\,,  &  R_n &:= \rk(W_n) = r_n + 2h_n\,, \\
    G_n &:= \prod_{i=1}^{l_n} G_{n,i}\,, & I_n &:= \prod_{i=1}^{h_n} I_2(m_{n,i})\,,
    \\
    r_{\max} &:= \max\{\rk(G_{n,1}), \ldots, \rk(G_{n,l_n})\}\,, & \mathcal{R}_n^2 &:= 
    \sum_{i=1}^{l_n} \rk(G_{n,i})^3\,, \\
     m_{\max} &:= \max\{m_{n,1}, \ldots, m_{n,h_n}\}\,, & \mathcal{M}_n^2 &:= \sum_{i=1}^{h_n} m_{n,i}^2\,. 
\end{align*}
Furthermore, we write $X_\inv^G$ and $X_\inv^I$ for the
number of inversions in the classical Weyl part and the dihedral part of $W_n$, respectively. As $\rk(W_n) = r_n + 2h_n,$ the growth condition is that at least one of $r_n \succ \log(k_n)^3$ or $h_n \succ \log(k_n)^3$ holds, i.e., $\log(k_n)^3 \prec \max\{r_n, h_n\}$.
Regardless of how $G_n$ is composed, Remark~\ref{rem:orders} tells us that 
\begin{align*}
    \E(X_\inv^G) &= \Theta(r_n^2)\,, & \Var(X_\inv^G) &= \Theta(r_n^3)\,,
\end{align*}
Combining this with Remark~\ref{remarkdihedral}, we obtain $\Var(X_\inv) = \Theta(\mathcal{R}_n^2 + \mathcal{M}_n^2)$.
By Theorem~\ref{thm:factorGF},
\begin{align*}
    \G_\inv(W_n;x) &= \prod_{i=1}^{R_n} (1 + x + \ldots + x^{d_i-1})\,,
\end{align*}
where the degrees $d_i$ encompass the degrees of the classical Weyl group parts (bounded by $2r_{\max}$), $h_n$ twos, and the numbers $m_{n1}, \ldots, m_{nh_n}$ (bounded by $m_{\max}$). For such composed groups, the sufficient condition \eqref{eq:growth1} for the Gumbel behavior of $X_\inv$ is
\begin{equation}\label{eq:maxBehave}
    \max\{n_{\max}, m_{\max}\} = O\left(\sqrt{R_n^{-1}(\mathcal{R}_n^2 + \mathcal{M}_n^2)}\right).
\end{equation}
These observations are summarized as follows:

\begin{theorem} \label{thm:mainWeylDihedral}Let $W_n = G_n \times I_n$ be a sequence of finite Coxeter groups, where the classical components are pooled in $G_n$ and the dihedral components are pooled in $I_n$. Let $k_n$ be a sequence of integers satisfying $k_n = \exp\Bigl(o\bigl(\max\{r_n \vee h_n\}^{1/3}\bigr)\Bigr)$. Let $(X_{nj})_{j=1,\ldots,k_n}$ be a row-wise i.i.d.\ triangular array with $X_{n1} \overset{\D}{=} X_\inv$ and let $M_n := \max\{X_{n1}, \ldots, X_{nk_n}\}$. Let $a_n, b_n$ be as in Theorem~\ref{thm:mainWeyl}. If condition \eqref{eq:maxBehave} holds, then 
  \[
    \PP(M_n \leq a_nx + b_n) \longrightarrow \exp\bigl(-\exp(-x)\bigr) \quad \forall x \in \R.
  \]
\end{theorem}

In the case of direct products consisting only of dihedral groups, i.e., $G_n = \emptyset$ and $W_n = \prod_{i=1}^{h_n} I_2(m_{n,i}),$ the statement of Theorem~\ref{thm:mainWeylDihedral} is simplified as follows.

\begin{corollary} \label{corollary4.8} 
Let $W_n = \prod_{i=1}^{h_n} I_2(m_{n,i})$ be a product of dihedral groups and $k_n = \exp\Bigl(o\bigl(h_n^{1/3}\bigr)\Bigr)$. Let $(X_{nj})_{j=1,\ldots,k_n},$ $M_n,$ $a_n,$ $b_n$ be as in Theorem \ref{thm:mainWeylDihedral}. If $m_{\max} \preccurlyeq h_n^{-1/2}\mathcal{M}_n,$ then $\PP(M_n \leq a_nx + b_n) \longrightarrow \exp\bigl(-\exp(-x)\bigr)$ $\forall x \in \R$.
\end{corollary}

\begin{remark}
  The condition $m_{\max} \preccurlyeq h_n^{-1/2}\mathcal{M}_n$ in Corollary~\ref{corollary4.8} is not trivial.  Writing the orders of the dihedral group as a vector
  $\mathbf{m}_n = (m_{n,1}, \ldots, m_{n,h_n})$, we get
  \[
    \|\mathbf{m}_n\|_\infty \preccurlyeq \frac{1}{\sqrt{n}}\|\mathbf{m}_n\|_2,
  \]
  where, as usual, $\| \negmedspace \cdot \negmedspace\|_{\infty}$ is the maximum norm and $\| \negmedspace \cdot \negmedspace \|_{2}$ is
  the euclidean norm.  Since
  $\|\mathbf{m}_n\|_\infty \geq \frac{1}{\sqrt{n}}\|\mathbf{m}_n\|_2$
  always holds, the condition can be stated more
  precisely as
  \[
    \|\mathbf{m}_n\|_\infty = \Theta\left(\frac{1}{\sqrt{n}}\|\mathbf{m}_n\|_2\right).
  \]
\end{remark}

\textbf{Remark.}\footnote{Added in the final arXiv version, not in published version.}
For the EVLT for $X_\des$ on sequences of mixed products of Coxeter groups, we
require $\Var(X_\des) = \Theta(r_n) + \sum_{i=1}^{h_n} m_{n,i}^{-1}
\overset{!}{=} \Theta(R_n)$ according to Theorem~\ref{thm4.5}. This is in
particular satisfied if all $m_{n,i}$ are uniformly bounded.

\section{Universal extreme value limit theorem for triangular arrays}
\label{sec:BerryEsseen}

As seen in the previous sections, large deviations theory can be employed to derive extreme value limit theorems for triangular arrays of exponential length. However, these tools require assumptions that are not satisfied in many situations. If the number of samples in the rows of the triangular arrays is strongly reduced, then the Gumbel extreme value limit can already be derived from the Berry--Esseen bound in the CLT. This allows to obtain a weaker but universal version of Theorem~\ref{thm:mainWeyl} for a very general class of families of distributions.

\begin{theorem} \label{thm5.1}
{Let $F_1, F_2, \ldots$ be a sequence of distributions which satisfy the Berry--Esseen bound}
\[
    {\sup_{x \in \R} \left|\frac{F_n(x) - \E(F_n)}{\sigma(F_n)} - \Phi(x)\right| = O(n^{-1/2})\,,}
\]
{where $\Phi$ is the CDF of $\NN(0,1)$. Let $(X_{nj})_{j=1,\ldots,k_n}$ be a triangular array with $X_{n1} \sim F_n$ and let $M_n, \alpha_n, \beta_n, a_n, b_n$ be as in Theorem~\ref{thm:mainWeyl}. If $k_n = O(n^\varepsilon)$ for some $\varepsilon < 1/2,$ then}
\[
    {\PP(M_n \leq a_nx + b_n) \longrightarrow \exp(-\exp(-x))\,.}
\]
\end{theorem}

\begin{proof}
{Let $Y_n := \sigma(X_{n1})^{-1}\left(X_{n1} - \E(X_{n1})\right)$ and $\mathcal{N} \sim \NN(0,1)$. Then, the Berry--Esseen bound is equivalent to}
\begin{equation}
    {\sup_{x \in \R} |\PP(Y_n > x) - \PP(\mathcal{N} > x)| = O(n^{-1/2})\,.} \label{3.11}
\end{equation}
{Now, replace $x$ with $x_n := \alpha_nx + \beta_n$ for fixed $x$. For monotonicity reasons, we can also assume that $k_n = \Omega(n^\delta)$ for some $\delta > 0$. From Mill's Ratio (see \cite{mills1926table}), we can deduce}
\begin{align*}
    {\PP(\mathcal{N} > x_n)}~& {= 1 - \Phi(\alpha_n x + \beta_n) \sim \frac{1}{\alpha_n x + \beta_n}\varphi(\alpha_n x + \beta_n)} \\
    &{= O\left(\frac{1}{\sqrt{\log(n)}}\right)\varphi\left(\frac{x}{\sqrt{2\varepsilon\log(n)}} + \sqrt{2\varepsilon\log(n)} - \frac{\log(4\pi\varepsilon\log(n))}{2\sqrt{2\varepsilon\log(n)}}\right)} \\
    &{= O\left(\frac{1}{\sqrt{\log(n)}}\right)\exp\left(-\varepsilon \log(n) - \frac{1}{2}\log(4\pi\varepsilon\log(n)) + O\left(\frac{\log(\log(n))^2}{\log(n)}\right)\right)} \\
    &{= O\left(\frac{1}{\sqrt{\log(n)}}\right)n^{-\varepsilon}(1 + o(1))\,,}
\end{align*}
{from which it follows that $\PP(\mathcal{N} > x_n) \succ n^{-1/2}$. From here, the proof continues the same way as in Theorem~\ref{thm:mainWeyl}.}
\end{proof}

\noindent This universal theorem strongly narrows down the permutation statistics and other distributions $(F_n)_{n \in \N}$ for which a triangular array of any size is not attracted to the Gumbel distribution.

\section{Outlook}
\label{sec:outlook}


There are several interesting ways to extend these results to other statistics and settings. However, the requirements of the specific proof methods based on Theorems~\ref{thmfeller2} and~\ref{thm3.2.4} are an obstacle to such extensions. For these methods, it is essential to know a factorization of the generating function, since it corresponds to an independent sum decomposition of the permutation statistic. Furthermore, the variances must be of appropriate magnitude to satisfy the control condition $|X_k| < \lambda_ns_n$ in Theorem~\ref{thmfeller2}. While CLTs can be proved without additive decompositions (e.g., those on the \textit{two-sided Eulerian statistic} $X_T(w) := X_\des(w) + X_\des(w^{-1})$ in \cite{bruck2019central}), they are crucial for our analysis of extreme values. 

An elaborate list of permutation statistics is provided within the database \cite{FindStat}. Many of these are asymptotically normal and satisfy the Berry--Esseen bound, which gives a Gumbel statement by Theorem~\ref{thm5.1} with a low bound of $k_n$. In each of these cases, it is an open question to obtain a subexponential bound of $k_n,$ or at least one that permits the uniform triangular array $(X_{nj})_{j=1,\ldots,n}$.


There is also interest in joint distributions of two or more permutation statistics, e.g., $(\inv(w), \des(w))$ or $(\des(w), \des(w^{-1}))$. The main challenge here is the dependence structure between the components, which means that new methods are necessary. Since the first posting of this paper, this problem has been addressed in \cite{dorr2023extremes}. Similar challenges arise when investigating the numbers of inversions or descents within other structures and distributions on permutation groups, such as conjugacy classes \cite{fulman1998distribution, kim2020central}, multisets or the Mallows distribution~\cite{he2022central}. 

A more general concept combining inversions and descents is that of \textit{$d$-inversions} and \textit{$d$-descents}. This concept was originally introduced only for permutation groups $A_n$. Inversions compare all pairs of indices, while descents compare only adjacent indices. Now, generalized $d$-inversions compare indices of distance at most $d$, with $d < n$ fixed, while $d$-descents compare indices of distance exactly $d$. A CLT was proved by Bona~\cite{bona2007generalized} using a dependency graph criterion on indicator random variables $Y_{ij} := \textbf{1}\{(i,j)$ forms a $d$-inversion$\}$. A generalization to signed and even-signed permutations can be achieved by transferring from index pairs to roots derived from the corresponding standard basis vectors (see Meier \& Stump~\cite{meier2022central}). Again, knowledge of generating functions is missing, but the extreme value theory of $d$-inversions is still interesting.

The proofs of large deviation theorems such as Theorems~\ref{thmfeller2} and~\ref{thm3.2.4} are very laborious. Besides, sum representations of permutation statistics with dependent summands are more readily available and proving dependence conditions may be more feasible than factorizations of generating functions. Therefore, it may be worth investigating if the independence condition in Theorem~\ref{thmfeller2} can be relaxed to weak dependence conditions. 

\section*{Acknowledgements}
The authors thank Claudia Kirch, Anja Janßen, and Christian Stump for many valuable comments on early versions of the manuscript. This work is funded by the Deutsche Forschungsgemeinschaft (314838170, GRK 2297, MathCoRe).

\bibliographystyle{amsplain}
\bibliography{bibliography.bib}

\bigskip \medskip

\noindent
\footnotesize {\bf Authors' addresses:}
\smallskip

\noindent Philip Dörr, OvGU Magdeburg, Germany,
{\tt philip.doerr@ovgu.de}

\noindent Thomas Kahle, OvGU Magdeburg, Germany,
{\tt thomas.kahle@ovgu.de}

\end{document}